\documentclass[12pt]{amsart}

\usepackage{a4wide,hyperref,amscd,amsmath}

\title[Second Lyapunov exponent]{On the second Lyapunov exponent of some multidimensional continued fraction algorithms}

\author[V.~Berth\'e]{Val\'erie Berth\'e}
\author[W.~Steiner]{Wolfgang Steiner}
\address{Universit\'e de Paris, IRIF, CNRS, F--75013 Paris, FRANCE}
\email{berthe@irif.fr, steiner@irif.fr}
\author[J. M. Thuswaldner]{J\"org M. Thuswaldner}
\address{Chair of Mathematics and Statistics, University of Leoben, A--8700 Leoben, AUSTRIA}
\email{joerg.thuswaldner@unileoben.ac.at}

\thanks{This work was supported by the Agence Nationale de la Recherche through the project Codys (ANR-18-CE40-0007).
The third author was supported by the projects FWF P27050 and FWF P29910 granted by the Austrian Science Fund and by project FWF/RSF I3466 granted by the Austrian Science Fund and the Russian Science Foundation. Part of this work has been done while the three authors were visiting the Erwin Schr\"odinger Institute in Vienna.}

\newtheorem{lemma}{Lemma}[section]
\newtheorem{theorem}[lemma]{Theorem}

\newtheorem{proposition}[lemma]{Proposition}

\theoremstyle{definition}

\theoremstyle{remark}

\newtheorem{remark}[lemma]{Remark}

\numberwithin{equation}{section}

\newcommand{\bx}{\mathbf{x}}
\newcommand{\by}{\mathbf{y}}
\newcommand{\bv}{\mathbf{v}}

\begin{document} 
\begin{abstract}
We study the strong convergence of certain multidimensional continued fraction algorithms. In particular, in the two- and three-dimensional  case, we prove that the second Lyapunov exponent of Selmer's algorithm is negative and bound it away from zero. Moreover, we give heuristic results on several other continued fraction algorithms. Our results indicate that all classical multidimensional continued fraction algorithms cease to be strongly convergent for high dimensions. The only exception seems to be the Arnoux--Rauzy algorithm which, however, is defined only on a set of measure zero.
\end{abstract}

\maketitle
\section{Introduction}
In the present paper we study strong convergence properties of multidimensional continued fraction algorithms. In particular, we give results and numerical studies for the second Lyapunov exponent of such algorithms. One of our  main objects is \emph{Selmer's algorithm}, which attracted a lot of interest in the recent years, mainly because of its relation to an (unordered) continued fraction algorithm  defined by Cassaigne in 2015. This algorithm, now called \emph{Cassaigne algorithm}, was studied in the context of  word combinatorics by Cassaigne, Labb\'e, and Leroy in~\cite{CLL:17}  where it was shown to be conjugate to Selmer's algorithm. Other properties of Selmer's algorithm have been studied in \cite{AL18,Bruin:15,Bruin:19_Erratum,HT:09,FS:19,Schweiger:01,Schweiger:04}.

The first results on the second Lyapunov exponent of Selmer's algorithm $A_S$ in dimension $d=2$ are due to Schweiger~\cite{Schweiger:01,Schweiger:04}, who proved strong convergence (see Section~\ref{sec:cf} for a definition) almost everywhere. Nakaishi~\cite{Nakaishi06} strengthened this result by showing that the second Lyapunov exponent $\lambda_2(A_S)$ satisfies  $\lambda_2(A_S) < 0$. Negativity of $\lambda_2(A_S)$ was conjectured already by Baldwin~\cite{Baldwin:92}, where Selmer's algorithm is called \emph{generalized mediant algorithm}, GMA for short (see also \cite{Baldwin:92b}; in particular, $e^{\lambda_2(A_S)}$ is numerically calculated in \cite[Table~I on p.~1522]{Baldwin:92}).  Labb\'e~\cite{Lab15} heuristically calculated the Lyapunov exponents for the Cassaigne and Selmer algorithms (for $d=2$); it is actually the equality of these values that indicated the conjugacy of the algorithms. We mention that Bruin, Fokkink, and Kraaikamp~\cite{Bruin:15} give a thorough study of Selmer's algorithm for dimensions $d \ge 2$; however, their proof of the fact that $\lambda_2(A_S)<0$ is incomplete~\cite{Bruin:19_Erratum}.  The simplicity of the Lyapunov spectrum of the Cassaigne algorithm is  proved by Fougeron and Skripchenko \cite{FS:19}; see also \cite{HT:09}. Heuristic calculations for the second Lyapunov exponent of other algorithms are  also provided by Baladi and Nogueira~\cite{Baladi-Nogueira:96}; see also~\cite{Nakaishi02}.

The proof of the negativity of the second Lyapunov exponent of Selmer's algorithm in dimension $d=2$ provided by Nakaishi~\cite{Nakaishi06} is intricate. 
In the present paper we  provide a simple proof for the fact that $\lambda_2(A_S)<0$ for $d=2$ which is based on ideas going back to Lagarias~\cite{Lagarias:93}  as well as Hardcastle and Khanin~\cite{Hardcastle:02,HK:02}. Moreover, we show that the matrices associated with the two-dimensional Selmer algorithm are Pisot whenever they are primitive, and we give a strictly negative upper bound for $\lambda_2(A_S)$.  For $d=3$, using extensive computer calculations (which yield exact results due to an appropriate error handling) we are able to prove  that the second Lyapunov exponent is negative as well. Again, we even provide a strictly negative upper bound for it. For higher dimensions we provide heuristic results. These results indicate that Selmer's algorithm is no longer strongly convergent for dimensions $d\ge 4$.

Another aim of the present paper is to provide numerical calculations in order to obtain heuristic estimates for the second Lyapunov exponent of other well-known continued fraction algorithms. In particular, we consider the Brun algorithm, the Jacobi--Perron algorithm,  the triangle map, and  a new algorithm which is ``in between'' the Arnoux--Rauzy algorithm and Brun's algorithm. It is interesting to see that apart from the Arnoux--Rauzy algorithm, which is strongly convergent in each dimension $d\ge 2$ (see~\cite{AD15}), all algorithms seem to be no longer strongly convergent for high dimensions. Since the Arnoux--Rauzy algorithm is defined only on a set of zero measure (the so-called Rauzy gasket, see~\cite{AS13,AHS:16}), we are not aware of any multi-dimensional continued fraction algorithm such as defined in Section \ref{sec:cf} which acts on a set of positive measure and is strongly convergent in all dimensions.
It was widely expected that the uniform approximation exponent, when it can  be expressed in terms of  the first  and second Lyapunov exponents of the algorithm  $A$ as $1 - \frac{\lambda_2(A)}{\lambda_1(A)}$ (see \cite[Theorem 1]{Lagarias:93})  would be larger than~$1$ (and strictly smaller than Dirichlet's bound $1+\frac{1}{d}$) for all $d \geq 2$; see e.g.\ \cite{Lagarias:93}. 
Our experimental studies indicate that this conjecture might not be true. 

Let us sketch the contents of this paper.
The formalism of multidimensional continued fraction algorithms considered in the present paper is recalled in Section~\ref{sec:cf} together with the conditions given by Lagarias~\cite{Lagarias:93}. The second Lyapunov exponent is discussed in Section~\ref{sec:second}. 
We deal with the connections with the Paley--Ursell inequality in  Section~\ref{sec:PU}.
After that we consider the Selmer algorithm in Section~\ref{sec:Selmer}, the Brun algorithm in Section~\ref{sec:Brun}, the Jacobi--Perron algorithm in Section~\ref{sec:JP},  a new algorithm inspired by the Arnoux--Rauzy algorithm in Section~\ref{sec:Algo}, and the   triangle map in Section~\ref{sec:Garrity}. 
Comparisons between these algorithms are provided in Section~\ref{sec:heuristic}. In the appendix we comment on the error handling needed for the floating point calculations used for the estimation of $\lambda_2(A_S)$ for $d\in\{2,3\}$.
\medskip

\paragraph{\bf Acknowledgment}
We   warmly   thank   S\'ebastien Labb\'e for  his  help with  numerical simulations.

\section{Multidimensional continued fraction algorithms}\label{sec:cf}
We first introduce the formalism of  multidimensional continued fraction algorithms  that will be used in the sequel. Observe that the algorithms we are dealing with in this paper mainly act  on sets of vectors whose entries are ordered (the only exception being the Jacobi--Perron  algorithm considered in Section~\ref{sec:JP}).
A~$d$-dimensional algorithm acts on a subset of the real vector space $\mathbb{R}^d$ for its renormalized version and on a subset of the real projective space $\mathbb{P}^d$ for its homogeneous version.
More precisely, for given $d \ge 2$ let 
\begin{align*}
\Lambda & = \{(y_0, y_1, \ldots ,y_d) \in \mathbb{R}^{d+1}\setminus\{\mathbf{0}\}:\, y_0 \ge y_1 \ge \cdots \ge y_d \ge 0\}, \\
\Delta & = \{(x_1, \ldots, x_d) \in \mathbb{R}^d:\, 1 \ge x_1 \ge \cdots \ge x_d \ge 0\}, 
\end{align*}
and the mappings
\begin{equation}\label{eq:iotakappa}
\begin{split}
\iota: &\ \Delta \to \Lambda, \quad (x_1,\ldots,x_d) \mapsto (1,x_1,\ldots,x_d),
\\
\kappa: &\ \Lambda \to \Delta, \quad (y_0, \ldots, y_d) \mapsto \Big(\frac{y_1}{y_0},\ldots,\frac{y_d}{y_0}\Big).
\end{split}
\end{equation}
Let the \emph{multidimensional continued fraction algorithm}
\[
A:\, \Delta \to \mathrm{GL}(d+1,\mathbb{Z}) 
\]
be defined in a way that the \emph{homogeneous version of the (ordered) multidimensional continued fraction algorithm}  ($\by=(y_0,\ldots,y_d)$)
\[
L:\, \Lambda \to \Lambda, \qquad \by \mapsto \by A(\kappa(\by))^{-1} = \by A\Big(\frac{y_1}{y_0},\ldots,\frac{y_d}{y_0}\Big)^{-1}
\]
is well defined (i.e., $L$ maps $\Lambda$ into itself). The \emph{projective version of the (ordered) multidimensional continued fraction algorithm $T$} is then defined by the commutative diagram
\[
\begin{CD}
\Lambda     @>L>>  \Lambda
\\@VV\kappa V        @VV\kappa V
\\\Delta     @>T>>  \Delta
\end{CD}
\]
We work with row vectors in the definition of the mappings $L$ and $T$ because this entails that
\[
A^{(n)}(\mathbf{x}) = A(T^{n-1}\mathbf{x}) \cdots \,A(T\mathbf{x})\,A(\mathbf{x})
\]
is a \emph{linear cocycle} which we shall call the \emph{cocycle associated with~$A$} (or just \emph{the cocycle $A$}). 
Indeed, $A^{(n)}$ fulfills the \emph{cocycle property} 
\begin{equation}\label{eq:cocA}
A^{(m+n)}(\bx) = A^{(m)}(T^n \bx) A^{(n)}(\bx).
\end{equation}
This cocycle produces the $d+1$ sequences of  rational convergents that are aimed to converge to~$\mathbf{x}$. Indeed, writing
\begin{equation} \label{e:An}
A^{(n)}(\mathbf{x}) = \begin{pmatrix}q_0^{(n)} & p_{0,1}^{(n)} & \cdots & p_{0,d}^{(n)} \\ q_1^{(n)} & p_{1,1}^{(n)} & \cdots & p_{1,d}^{(n)} \\ \vdots & \vdots & \ddots & \vdots \\ q_d^{(n)} & p_{d,1}^{(n)} & \cdots & p_{d,d}^{(n)}\end{pmatrix}
\end{equation}
and $\mathbf{p}_i^{(n)} = (p_{i,1}^{(n)}, \ldots, p_{i,d}^{(n)})$, we consider the convergence  of  $\lim_{n\to\infty} \mathbf{p}_i^{(n)}/q_i^{(n)} $  to $ \mathbf{x}$, $0 \le i \le d$.
The convergence is said to be  \emph{weak} if $\lim_{n\to\infty} \mathbf{p}_i^{(n)}/q_i^{(n)}= \mathbf{x}$ for all $i$ with $0 \le i \le d$, and \emph{strong} if $\lim_{n\to\infty} \vert \mathbf{p}_i^{(n)} - q_i^{(n)} \mathbf{x} \vert = 0$ for all  $0 \le i \le d$.

Since we focus on the action of the matrices produced by the algorithm on the  orthogonal space $\iota(\mathbf{x})^\bot$ of $\iota(\mathbf{x})$, we use left-multiplication for the description of the linear action in order to simplify notation and to avoid the use of the transpose.

Throughout this paper we suppose that a multidimensional continued fraction algorithm satisfies the following conditions which go back to Lagarias~\cite{Lagarias:93}. Similar to \cite{FS:19} we just explain them briefly and refer to Lagarias' paper for details. 
\begin{description}
\item[(H1) Ergodicity] 
The map $T$ admits  an ergodic  invariant probability measure~$\mu$ that is    absolutely continuous with respect to Lebesgue measure on $\Delta$.  
\item[(H2) Covering Property]  The map $T$ is piecewise continuous with non-vanishing Jacobian almost everywhere.
\item[(H3) Semi-weak convergence]  This is a mixing condition for $T$ which implies weak convergence. If $T$ admits a Markov partition it can be checked by making sure that the cylinders of the Markov partition decrease geometrically. For some examples this is worked out in \cite{Lagarias:93}.
See \cite{FS:19} for   a sufficient condition expressed   in terms of   the existence of   a special acceleration  providing  a simplex on which the induced algorithm is  uniformly expanding.
\item[(H4) Boundedness] This is log-integrability of the cocycle $A$, i.e.,
finiteness of the expectation of $\log( \max (\Vert A\Vert,1)$. This is necessary in order to apply the Oseledets Theorem.

\item[(H5) Partial quotient mixing] This condition says that the expectation of  the number $n$ for which $A^{(n)}(\mathbf{x})$ becomes a strictly positive matrix is finite.
\end{description}

Throughout the paper the Lyapunov exponents of the cocycle~$A$ are denoted as
\[
\lambda_1(A) \ge \lambda_2(A) \ge \cdots \ge  \lambda_{d+1}(A).
\]

Our motivation for studying the second Lyapunov exponent is due to the fact that it is related to the \emph{uniform approximation exponent}, a quantity that estimates the rate of convergence of a continued fraction algorithm. We recall the definition of this object; see also \cite[Definition~38]{Schweiger:00} or \cite[Section~1]{Lagarias:93}. 

Let $A$ be a multidimensional continued fraction algorithm with cocycle $A^{(n)}$ given as in \eqref{e:An}. For $\bx\in\Delta$ and $i\in\{0,\ldots,d\}$ set (for an arbitrary norm $\Vert\cdot\Vert$ in $\mathbb{R}^d$)
\[
\eta_A^*(\bx,i) = \sup
\bigg\{
\delta >0 \;:\; \exists\, n_0=n_0(\bx,i,\delta) \in\mathbb{N} \hbox{ s.\ t.\ }\forall \, 
n\ge n_0,\; \bigg\Vert  \bx-\frac{\mathbf{p}_i^{(n)}}{q_i^{(n)}} \bigg\Vert < (q_i^{(n)})^{-\delta}
\bigg\}.
\]
Then
\[
\eta^*_A(\bx) = \min_{0\le i\le d} \eta_A^*(\bx,i)
\]
is called the \emph{uniform approximation exponent} for $\bx$ using the algorithm $A$.

The following result  links the second Lyapunov exponent with the uniform approximation exponent; see \cite[Theorem~1]{HK00} for a variant of this result and \cite[Proposition~4]{Baldwin:92}.

\begin{proposition}[{\cite[Theorem~4.1]{Lagarias:93}}]\label{prop:lag}
Let $\eta^*_A$ be the uniform approximation exponent of a $d$-dimensional multidimensional continued fraction algorithm $A$ satisfying conditions (H1) to (H5). We have $\lambda_1(A) > \lambda_2(A)$ and 
\[
\eta^*_A(\mathbf{x}) = 1- \frac{\lambda_2(A)}{\lambda_1(A)}
\]
holds for almost all $\mathbf{x}\in\Delta$. In particular, if $\lambda_2(A) < 0$ then $A$ is a.e.\ strongly convergent.
\end{proposition}

We wish to show that $\lambda_2(A) < 0$ for various  classical  multidimensional continued fraction algorithms $A$. To apply Proposition~\ref{prop:lag} we need to make sure that these algorithms satisfy conditions (H1) to (H5). As we will see in the subsequent sections, these conditions are known to hold for most of the algorithms we discuss (and will be treated for the remaining ones in a forthcoming paper).

\section{The second Lyapunov exponent}\label{sec:second}
The action of a   continued fraction algorithm is given by a matrix~$A$  acting  by  left-multiplication  on some direction.
To understand the quality of approximation, it is useful to work on the orthogonal of this direction.
The action on the orthogonal is then given by the matrix~$A$ acting by right-multiplication.
We are thus interested  in the action of the matrix~$A$ and  of the associated cocyle  on a restricted hyperplane.
By choosing a suitable basis of this hyperplane, the action of the algorithm is  then  described as a matrix that involves the usual  differences that have the form $q_n  x -p_n$ in the one-dimensional case, where $p_n/q_n$ are the convergents of $x$.

In order to give estimates of the second Lyapunov exponent of a multidimensional continued fraction algorithm 
we follow the ideas of Hardcastle and Khanin~\cite{Hardcastle:02,HK:02} who built on the work of Lagarias~\cite{Lagarias:93}.  

Since $T$ is ergodic by (H1), the Lyapunov exponents of~$A$ are the same for almost all $\mathbf{x} \in \Delta$ w.r.t.\ the invariant measure of~$T$. Under the conditions of Proposition~\ref{prop:lag}, the Oseledets Theorem gives, for generic $\mathbf{x} \in \Delta$,
\[
\lim_{n\to \infty} \frac1n \log\Vert A^{(n)}(\mathbf{x})\, \mathbf{v} \Vert \le  \lambda_2(A) \quad \mbox{if and only if} \quad \mathbf{v} \in \iota(\mathbf{x})^\perp,
\]
where $\mathbf{y}^\perp = \{\mathbf{v} \in \mathbb{R}^{d+1}:\, \mathbf{y} \mathbf{v} = 0\}$; see \cite[Theorem 4.1]{Lagarias:93} for more details. 
Note that $A^{(n)}(\mathbf{x})\, \iota(\mathbf{x})^\perp = \iota(T^n\mathbf{x})^\perp$.
Using the notation in~\eqref{e:An}, the matrix $D^{(n)}(\mathbf{x})$ defined~by
\begin{equation}\label{eq:Dn}
D^{(n)}(x_1,x_2,\ldots,x_d) = \begin{pmatrix}p_{1,1}^{(n)} - q_1^{(n)} x_1 & \cdots & p_{1,d}^{(n)} - q_1^{(n)} x_d \\ \vdots & \ddots & \vdots \\ p_{d,1}^{(n)} - q_d^{(n)} x_1 & \cdots & p_{d,d}^{(n)} - q_d^{(n)} x_d\end{pmatrix}
\end{equation}
is a cocycle of~$T$ \cite[Proposition~4.1]{HK:02} satisfying $\lambda_2(A) = \lambda_1(D)$ \cite[Lemma~3.1]{HK:02}.
For the sake of self-containedness, we prove the cocycle property here (see (\ref{eq:cocA}))  and the equality of the Lyapunov exponents in Remark~\ref{r:lambdaDA} below.
We have 
\begin{equation}\label{eq:DnAn}
D^{(n)}(\mathbf{x}) = \Pi\, A^{(n)}(\mathbf{x}) H(\mathbf{x}),
\end{equation}
with
\[
\Pi = \begin{pmatrix}0 & 1 & 0 & \cdots & 0 \\ \vdots & \ddots & \ddots & \ddots & \vdots \\ \vdots & & \ddots & 1 & 0 \\ 0 & \cdots & \cdots & 0 & 1\end{pmatrix}, \quad 
H(x_1,x_2,\ldots,x_d) = \begin{pmatrix}-x_1 & -x_2 & \cdots & -x_d \\ 1 & 0 & \cdots & 0 \\ 0 & 1 & \ddots & \vdots \\ \vdots & \ddots & \ddots & 0 \\ 0 & \cdots & 0 & 1\end{pmatrix}.
\]
and we also have that 
\begin{equation}\label{eq:HPweg}
H(T^n \mathbf{x})\, \Pi\, A^{(n)}(\mathbf{x}) H(\mathbf{x}) = A^{(n)}(\mathbf{x}) H(\mathbf{x}).
\end{equation}
Indeed, the matrix $I_{d+1} - H(T^n \mathbf{x})\, \Pi$ is zero except for the first row, which is~$\iota(T^n \mathbf{x})$, and we have $\iota(T^n \mathbf{x}) A^{(n)}(\mathbf{x}) H(\mathbf{x}) = \iota(\mathbf{x}) H(\mathbf{x}) = \mathbf{0}$.
Using \eqref{eq:DnAn}, \eqref{eq:HPweg}, and \eqref{eq:cocA} we obtain that
\[
D^{(m)}(T^n\mathbf{x}) D^{(n)}(\mathbf{x}) = \Pi\, A^{(m)}(T^n\mathbf{x}) A^{(n)}(\mathbf{x}) H(\mathbf{x}) = D^{(m+n)}(\mathbf{x}),
\]
thus $D^{(n)}(\mathbf{x})$ is a cocycle of~$T$. 

Therefore, it suffices to estimate the first Lyapunov exponent of the cocycle $D^{(n)}(\mathbf{x})$. This is convenient because it is usually easier to obtain  estimates for the first Lyapunov exponent of a cocycle than for the second one. 
As observed by Hardcastle and Khanin~\cite{Hardcastle:02,HK:02}, the Subadditive Ergodic Theorem yields that 
\begin{equation}\label{eq:int}
\lambda_2(A) = \lambda_1(D) = \inf_{n\in\mathbb{N}} \frac{1}{n} \int_{\Delta} \log \| D^{(n)}(\mathbf{x})\|\, \mathrm{d}\mu(\mathbf{x})
\end{equation}
for any matrix norm, see \cite[Lemma~3.3]{HK:02}. 
We note that the matrices $D^{(n)}(\bx)$ were first studied by Fujita, Ito, Keane, and Ohtsuki \cite{FIKO,IKO}.
Observe also  that strong convergence at a point $\mathbf{x}$ is equivalent to $\lim_{n\to\infty} \Vert  D^{(n)}(\mathbf{x})\Vert = 0$.
Indeed, $\lim_{n\to\infty} \Vert  D^{(n)}(\mathbf{x})\Vert = 0$ means that $\lim _{n\to\infty}  \vert p_{i,j}^{(n)} -q_i^{(n)} x_j \vert=0$ for $i,j\ge 1$, and we then use the orthogonality of the columns of $A^{(n)}(\mathbf{x}) H(\mathbf{x}) $ to $\iota (T^n \mathbf{x})$ to deduce that  $ \lim_{n\to\infty} \vert p_{0,j}^{(n)} -q_0^{(n)} x_j\vert = 0$ for $ j\geq 1$.

There  exist several methods for      providing  numerical estimates  for  the computation of the second Lyapunov exponent.
The  approach of  \cite{Baladi-Nogueira:96}, which  is inspired  by  \cite{JPS:87},  is based on   the decomposition  of matrices
 as a product of a unitary  matrix $ Q$ and an upper triangular matrix $R$.   
For low dimensions~$d$, we can  also evaluate the integrals in  (\ref{eq:int})  symbolically (using polylogarithms) with a computer algebra software such as \texttt{Mathematica} or use estimates for the measure~$\mu$ to show that $\lambda_2(A) < 0$ for some continued fraction algorithms, in particular for the Selmer algorithm.  
Indeed, the densities of invariant measures have simple particular forms; see e.g.\ \eqref{eq:mu} below.
For higher dimensions, these calculations take too much time and we can only make simulations of the behaviour of $D^{(n)}(\mathbf{x})$ for randomly chosen points~$\mathbf{x}$.  According to these simulations, 
it seems that we have $\lambda_2(A) > 0$ for all known continued fraction algorithms when $d$ gets large, contrary to conjectures of e.g.\ \cite{Lagarias:93,Hardcastle:02}.

\section{On the Paley--Ursell inequality}\label{sec:PU}

We recall that the notation $ f_n  \ll  g_n $ means that  there exists $C>0$ such that
$f_n \leq C g_n$  for all $n$. Let $A$ be a multidimensional continued fraction algorithm and let $D^{(n)}$ be as in \eqref{eq:Dn}. For certain algorithms $A$ we have
\begin{equation}\label{eq:Dnbound}
\Vert D^{(n)} (\mathbf{x}) \Vert \ll 1 \hbox{ uniformly for all } \mathbf{x}, 
\end{equation}
which is a form of the \emph{Paley--Ursell inequality}, going back to Paley and Ursell \cite{PU:30}. This inequality essentiall says that the second Lyapunov exponent of the algorithm is nonpositive. This inequality can be formulated in terms of an inequality for $D^{(n)}$ as in \eqref{eq:Dnbound}, or for the minors of size $2$ of the matrices~$A^{(n)}$ (see~Proposition~\ref{prop:PU} below).

Recall that  \eqref{eq:Dnbound} means that $| p_{i,j}^{(n)} - q_i ^{(n)}  x_j |  \ll 1$ for all $i,j\in\{1,\ldots, d\}$ uniformly in $ {\bf x}$. Thus the Paley--Ursell inequality is a statement on the quality of the approximation of a sequence of convergents. In this section we  discuss the relations between different forms of the Paley--Ursell inequality. In particular, we show that \eqref{eq:Dnbound} implies an inequality bounding the norm of the second exterior product of $A^{(n)}$ in terms of the norm of~$A^{(n)}$. 

We will see in Section~\ref{sec:Selmer} that \eqref{eq:Dnbound} holds for Selmer in dimension $d=2$. It also holds for Brun in dimension $d=2$ and for Arnoux--Rauzy for arbitrary dimension $d\ge 2$ according to Avila and Delecroix~\cite{AD15} and Remark~\ref{r:AD} below.
The original version in~\cite{PU:30} is proved for Jacobi--Perron in dimension $d=2$. 
In the form we state it below, it is contained in Broise and Guivarc'h~\cite{BG:01}. 

Contrary to the results we discussed in the previous section, the results of this section are true for \emph{all} $\mathbf{x}\in\Delta$ (except pathological cases when the algorithm  terminates  and is not defined). The price we have to pay for getting a result that is valid everywhere is that it is weaker than the metric results we expect to be true. Indeed, while Section~\ref{sec:second} is tailored to be the starting point for proving that $\lambda_2(A) < 0$ almost everywhere, inequality \eqref{eq:Dnbound} implies that $\lambda_2(A) \le 0$ {\em everywhere}. Moreover, \eqref{eq:Dnbound} is true for each time $n$ in an orbit and not only in the limit. 
 
In the following proposition $\wedge^2$ denotes the second exterior product.

\begin{proposition}\label{prop:PU}
Consider a multidimensional continued fraction algorithm  satisfying conditions (H1) to (H5).
If $\Vert D^{(n)} (\mathbf{x}) \Vert \ll 1$ holds uniformly in~$\mathbf{x}$,  then 
\begin{equation}\label{eq:PUA}
\Vert \wedge^2 A ^{(n)} (\mathbf{x}) \Vert \ll \Vert A ^{(n)}(\mathbf{x})\Vert
\end{equation}
holds uniformly in ${\mathbf x}$.
\end{proposition}

This result implies that $A^{(n)}(\mathbf{x})$ maps the unit sphere in $\mathbb{R}^n$ to an ellipsoid whose second largest semi-axis $\delta_2(A ^{(n)}(\mathbf{x}))$ is uniformly bounded in $\mathbf{x}\in\Delta$ and $ n\in \mathbb{N}$.  Moreover, since the elements of $\wedge^2 A ^{(n)} (\mathbf{x})$ are the $2\times 2$ minors of $A ^{(n)} (\mathbf{x})$ this inequality shows that the $2\times 2$ minors of $A ^{(n)} (\mathbf{x})$ cannot be much larger than its elements. 

To prove this result we need the following preparatory lemma. We write $\delta_i(M)$ for the $i$-th largest singular value of a $k\times k$ matrix~$M$ ($1\le i\le k$, $k\in \mathbb{N}$).  

\begin{lemma}\label{lem:singular}
The inequality 
\[
\delta_2(A ^{(n)}(\mathbf{x})) \ll \delta_1(D^{(n)}(\mathbf{x}))
\]
holds uniformly for all $\mathbf{x} \in\Delta$. 
\end{lemma}

\begin{proof}
Recall that $D^{(n)}(\mathbf{x})=\Pi A^{(n)}(\mathbf{x})H(\mathbf{x})$.  In order to estimate the singular values of $D^{(n)}(\mathbf{x})$, we map the unit ball  $\mathbb{S}^{d-1}$ in $\mathbb{R}^d$ step by step by the matrices $H(\mathbf{x})$, $A ^{(n)}(\mathbf{x})$, and $\Pi$, and keep track of the length of the semi-axes of the ellipsoids  which are deformed. The ellipsoid $H(\mathbf{x})\mathbb{S}^{d-1}$ is a subset of the hyperplane $\iota(\mathbf{x})^\bot$ whose semi-axes $\mathbf{a}_i^{(1)}$ satisfy $1\ll \Vert \mathbf{a}_i^{(1)} \Vert \ll 1$ ($1 \le i\le d$). 
By the definition of the singular values $\delta_i(A ^{(n)}(\mathbf{x}))$ ($1 \le i \le d+1$), this implies that the ellipse $A^{(n)}(\mathbf{x}) H(\mathbf{x})\mathbb{S}^{d-1} \subset \iota(T^n\mathbf{x})^\bot$ has semi-axes $\mathbf{a}_i^{(2)}$ satisfying 
\begin{equation}\label{eq:qi2est}
\Vert \mathbf{a}_i^{(2)} \Vert \gg \delta_{i+1}(A ^{(n)} (\mathbf{x})) \qquad (1 \le i \le d).
\end{equation}
It remains to apply the projection~$\Pi$. 
Since $\iota(T^n\mathbf{x})=(1,y_1,\ldots,y_d)$ with $|y_i|\le 1$ ($1 \le i \le d$), the angle between the hyperplanes $\iota(T^n\mathbf{x})^\bot$ and $(1,0,\ldots,0)^\bot$ of $\mathbb{R}^{d+1}$ is greater than $c>0$ for some constant $c$ not depending on $n$. Thus the projection $\Pi$ shrinks each vector $\mathbf{v}\in \iota(T^n\mathbf{x})^\bot$ by a factor which is greater than or equal to $\sin c$. Thus, because $A (\mathbf{x}) H(\mathbf{x})\mathbb{S}^{d-1} \subset \iota(T^n\mathbf{x})^\bot$ we get from $\eqref{eq:qi2est}$ that
\[
\delta_2(A ^{(n)} (\mathbf{x})) \ll \delta_1(\Pi A ^{(n)} (\mathbf{x})H(\mathbf{x})) = \delta_1(D^{(n)}(\mathbf{x})). \qedhere
\]
\end{proof}

We can now finish the proof of Proposition~\ref{prop:PU}.

\begin{proof}[Proof of Proposition~\ref{prop:PU}]
Suppose that $\Vert D^{(n)}(\mathbf{x}) \Vert \ll 1$ holds. Lemma~\ref{lem:singular} implies that
\begin{align*}
\Vert \wedge^2 A ^{(n)} (\mathbf{x}) \Vert_2 &= \delta_1(\wedge^2 A ^{(n)} (\mathbf{x})) = \delta_1(A ^{(n)} (\mathbf{x})) \delta_2(A ^{(n)} (\mathbf{x})) \\
& \ll \delta_1(A ^{(n)} (\mathbf{x})) \delta_1(D^{(n)}(\mathbf{x})) = \delta_1(A ^{(n)} (\mathbf{x}))\Vert D^{(n)}(\mathbf{x})  \Vert_2 \\
& \ll \Vert A ^{(n)} (\mathbf{x})\Vert_2, 
\end{align*}
where the implied constants do not depend on $\mathbf{x}$ and~$n$. The estimate in \eqref{eq:PUA} follows from this by the equivalence of norms. 
\end{proof}

We note that the converse of Proposition~\ref{eq:PUA} is not true in general. In particular, to get the converse, assumptions on the sequence of matrices $(A^{(n)}(\mathbf{x}))_n$ are needed in order to guarantee that all the quantities $q_n^{(i)}$ ($0\le i\le d$) are roughly of the same size for each~$n$ (as is true for instance for the Jacobi--Perron algorithm, see \cite[Section~5.2]{BG:01}); see  also Proposition \ref{prop:balanced} below.
More precisely, one says that the  {\em balancedness condition} holds for the sequence  $(A^{(n)} (\mathbf{x}))_n$
if  the (vector) norms of the lines of $A^{(n)} (\mathbf{x})$ are comparable (up to multiplicative constants) with the (matrix) norm of $A^{(n)} (\mathbf{x})$,  with these constants being uniform in $n$.

\begin{proposition} \label{prop:balanced}
Assume that  the   balancedness condition  hods for  $(A^{(n)} (\mathbf{x}))_n$. Then
\[
\Vert D^{(n)} (\mathbf{x}) \Vert  \Vert  A^{(n)} (\mathbf{x})  \Vert  \ll   \Vert \wedge ^{2}  A^{(n)} (\mathbf{x})  \Vert. 
\]
\end{proposition}

\begin{proof}
By definition, $\iota ( \mathbf{x})$ is equal to $\iota (T^n \mathbf{x}) A^{(n)}(\mathbf{x})$ divided by its first coordinate.
In other words, for  $1 \leq i \leq d$,    $$ {x}_i= \frac{   p _{0,i} ^{(n)} + p_{1,i} ^{(n)} x_1^{(n)}
 + \cdots+  p_{d,i}^{(n)}  x_d^{(n)}}{  q _{0}^{(n)} + q_{1} ^{(n)} x_1^{(n)}
 + \cdots+  q_{d}^{(n)}  x_d^{(n)}}.$$
Hence, for all  $1 \leq i,j \leq d$, one has
\[
\left \vert x_i -  \frac{ p_{j,i}^ {(n)} }{q_{j}^ {(n)} } \right \vert = \frac{ p_{0,i}^{(n)}   q_{j}^{(n)} -q_0^{(n)}  p_{j,i} ^{(n)}+ x_1^{(n)}  (p_{1,i} ^{(n)}q_{j}^{(n)} -
  q_1^{(n)} p _{j,i}^{(n)} ) + \cdots+  x_d^{(n)}  (p_{d,i} ^{(n)}q_{j}^{(n)} -
  q_d^{(n)} p _{j,i}^{(n)} ) } {(q _{0}^{(n)} + q_{1} ^{(n)} x_1^{(n)}
 + \cdots+  q_{d}^{(n)}  x_d^{(n)})  q_{j}^{(n)} },
\]
which  implies  together with the balancedness assumption that
\[
\Vert D^{(n)} (\mathbf{x}) \Vert \ll  \frac {d\, \Vert \wedge^{2}  A^{(n)}(\mathbf{x})  \Vert}{(d+1)\, \Vert A^{(n)} \Vert}. \qedhere
\]
\end{proof}

\begin{remark} \label{r:AD}
A condition similar to \eqref{eq:Dnbound} is used in Avila and Delecroix~\cite{AD15}, namely 
\[
\Vert A^{(n)}  \vert_{\iota(T^n\mathbf{x})^\bot} \Vert  \ll 1
\]
uniformly in $\mathbf{x}$.  
This implies \eqref{eq:Dnbound} and, hence, by Proposition~\ref{prop:PU} also \eqref{eq:PUA} . 
\end{remark}

\begin{remark} \label{r:lambdaDA}
For a multidimensional continued fraction algorithm satisfying conditions (H1) to (H5), we recover the fact that $\lambda_1 (D)=\lambda_2(A)$ from Propositions~\ref{prop:PU} and~\ref{prop:balanced}, by  using that the denominators $q_i^{(n)}$ grow at the same exponential rate, as observed in \cite{Lagarias:93}.
\end{remark}

\section{Selmer algorithm}\label{sec:Selmer} 
\subsection{Definition}\label{sec:defselmer} In its (ordered)  homogeneous form,
Selmer's algorithm is defined by subtracting the smallest element of a vector from the largest one and reordering the elements in the resulting vector; see Selmer~\cite{Selmer:61} or Schweiger~\cite[Chapter~7]{Schweiger:00}. Formally,
\[
T_S:\, \Delta \to \Delta, \quad  T_S(x_1,\ldots,x_d) = \kappa(\mathrm{ord}(1-x_d,x_1,x_2,\ldots,x_d)),
\]
where $\kappa$ is defined in \eqref{eq:iotakappa} and 
\[
\mathrm{ord}:\mathbb{R}^n\to\mathbb{R}^n
\]
orders the entries of its argument descendingly. Let
\[
A_S(\mathbf{x}) = \begin{cases}S_a & \mbox{if}\ \mathbf{x} \in \Delta_{S_a} := \{(x_1,\ldots,x_d) \in \Delta:\, 2x_d > 1\}, \\ S_b & \mbox{if}\ \mathbf{x} \in \Delta_{S_b} := \{(x_1,\ldots,x_d) \in \Delta:\, 2x_d < 1 \le x_{d-1}+x_d\},\end{cases}
\]
with 
\[
S_a = \begin{pmatrix}0&1&0&\cdots&0 \\ \vdots&\ddots&\ddots&\ddots&\vdots \\ 0&\cdots&0&1&0 \\ 1&0&\cdots&0&1 \\ 1&0&\cdots&0&0\end{pmatrix}, \qquad S_b = \begin{pmatrix}0&1&0&\cdots&0 \\ \vdots&\ddots&\ddots&\ddots&\vdots \\ 0&\cdots&0&1&0 \\ 1&0&\cdots&0&0 \\ 1&0&\cdots&0&1\end{pmatrix}.
\]
Since for almost all $\mathbf{x} \in \Delta$, we have $T_S^n \mathbf{x} \in \Delta_{S_a} \cup \Delta_{S_b}$ for all sufficiently large~$n$ (see \cite[Theorem~22]{Schweiger:00}), it suffices to consider the absorbing set $\Delta_{S_a} \cup \Delta_{S_b}$.
In all that follows, we do not care about the behaviour of $T_S$ on the boundary of $\Delta_{S_a}$ and $\Delta_{S_b}$ because we are interested only in metric results. The invariant measure of $T_S$ is
\begin{equation} \label{eq:mu}
\mathrm{d}\mu_S = c \frac{\mathrm{d}x_1}{x_1} \frac{\mathrm{d}x_2}{x_2} \cdots \frac{\mathrm{d}x_d}{x_d}
\end{equation}
on $\Delta_{S_a} \cup \Delta_{S_b}$, with normalizing constant $c$ such that $\mu_S(\Delta_{S_a} \cup \Delta_{S_b}) = 1$; see \cite[Theorem~22]{Schweiger:00}. 
As shown in \cite[Section~6]{Lagarias:93}, Selmer's algorithm satisfies the assumptions of Proposition~\ref{prop:lag} (in particular, it satisfies the assumptions (H1) to (H5)).

Observe that a  multiplicative version of Selmer's algorithm can also be considered; see e.g.\ \cite{Kops:12,Schweiger:04}. This algorithm is not an acceleration of the additive version. Moreover, it does not behave well in terms of convergence; see \cite[Section~2]{Schweiger:04}.

\subsection{Second Lyapunov exponent, $d=2$} \label{sec:second-lyap-expon}
As mentioned before, Nakaishi~\cite{Nakaishi06} gave an intricate proof of the fact that $\lambda_2(A_S) < 0$ for $d=2$; see also \cite{Schweiger:01}. 
We provide a very simple proof of this fact and, on top of this, we are able to bound $\lambda_2(A_S)$ away from~$0$. 
The following result should be compared to Labb\'e~\cite{Lab15}, who conjectures on the basis of computer experiments that $-0.07072$ is a good approximation to $\lambda_2(A_S)$ (and to Table~\ref{tab:0}, where we confirm this value by our computer estimates).

\begin{theorem} \label{thm:strongSelmer}
For $d=2$, the second Lyapunov exponent of the Selmer algorithm satisfies
\[
\lambda_2(A_S) < -0.052435991.
\]
In particular, for $d=2$ the Selmer algorithm is a.e.\ strongly convergent.
\end{theorem}

\begin{proof}
We have 
\[
S_a^2 = \begin{pmatrix}1&0&1\\1&1&0\\0&1&0\end{pmatrix},\ 
S_a S_b = \begin{pmatrix}1&0&0\\1&1&1\\0&1&0\end{pmatrix},\ 
S_b S_a = \begin{pmatrix}1&0&1\\0&1&0\\1&1&0\end{pmatrix},\ 
S_b^2 = \begin{pmatrix}1&0&0\\0&1&0\\1&1&1\end{pmatrix},
\]
and the corresponding matrices $D_S^{(2)}(x_1,x_2)$ are
\[
\begin{pmatrix}1-x_1&-x_2\\1&0\end{pmatrix},\ 
\begin{pmatrix}1-x_1&1-x_2\\1&0\end{pmatrix},\ 
\begin{pmatrix}1&0\\1-x_1&-x_2\end{pmatrix},\ 
\begin{pmatrix}1&0\\1-x_1&1-x_2\end{pmatrix}.
\]
Since $x_1+x_2>1>x_1>x_2>0$, we have thus $\|D_S^{(2)}(\mathbf{x})\|_\infty = 1$ for all $\mathbf{x} \in \Delta_{S_a} \cup \Delta_{S_b}$.
This already implies that $\lambda_1(D_S) \le 0$ by (\ref{eq:int}).

Moreover, this implies that $\|D_S^{(4)}(\mathbf{x})\|_\infty \le 1$ for all $\mathbf{x} \in \Delta_{S_a} \cup \Delta_{S_b}$.
We have
\[
(S_a S_b)^2 = \begin{pmatrix}1&0&0\\2&2&1\\1&1&1\end{pmatrix}, \quad \mbox{thus} \quad  
D_S^{(4)}(x_1,x_2) = \begin{pmatrix}2-2x_1&1-2x_2\\1-x_1&1-x_2\end{pmatrix}
\]
for $(x_1,x_2) \in \Delta_{S_b} \cap T_S^{-1} \Delta_{S_a} \cap T_S^{-2} \Delta_{S_b} \cap T_S^{-3} \Delta_{S_a}$, i.e., $(x_1,x_2)$ in the triangle with corners $(3/4,1/2)$, $(3/5,2/5)$, $(2/3,1/3)$. 
We have thus 
\[
\|D_S^{(4)}(3/4-\varepsilon,1/2-\varepsilon)\|_\infty = 3/4+2\varepsilon, 
\]
hence, $\lambda_1(D_S) \le \frac{1}{4} \int_{\Delta} \log \| D_S^{(4)}(\mathbf{x}) \|_\infty \mathrm{d}\mu_S(\mathbf{x}) < 0$.

To get better upper bounds for $\lambda_1(D_S)$, note that $A_S^{(n)}(\mathbf{x}) = M \in \{S_a,S_b\}^n$ for all $\mathbf{x}$ in the triangle
\begin{equation}\label{eq:DeltaM}
\Delta_M = \{\mathbf{x} \in \Delta:\, \iota(\mathbf{x}) \in \mathbb{R}\, \iota(\Delta_{S_a} \cup \Delta_{S_b})\, M\}.
\end{equation}
We have thus 
\begin{equation}\label{eq:lamest}
\lambda_1(D_S) \le \frac{1}{n} \sum_{M\in\{S_a,S_b\}^n} \mu_S(\Delta_M) \max_{\mathbf{x}\in \Delta_M} \log \|D_S^{(n)}(\mathbf{x})\|_\infty
\end{equation}
for all $n \ge 1$. The measure of $\Delta_M$ can be calculated using dilogarithms; here we  only need to  bound it  by
\begin{equation}\label{eq:muest}
\mu_S(\Delta_M) \ge \frac{12}{\pi^2} \min_{\mathbf{x}\in \Delta_M} \frac{1}{x_1x_2} \mathrm{Leb}(\Delta_M) \ge \frac{12}{\pi^2} \min_{\mathbf{x}\in \Delta_M} \frac{1}{x_1} \min_{\mathbf{x}\in \Delta_M} \frac{1}{x_2} \mathrm{Leb}(\Delta_M);
\end{equation}
note that $c = 12/\pi^2$ in the definition of $\mu_S$ for $d=2$. 
Since $\log \|D_S^{(2n)}(\mathbf{x})\|_\infty \le 0$ for all $\mathbf{x} \in \Delta_{S_a} \cup \Delta_{S_b}$, we obtain, by using \eqref{eq:muest} to estimate $\mu_S(\Delta_M)$ in \eqref{eq:lamest} and taking even powers of matrices,  that
\begin{equation}\label{eq:tcd2}
\lambda_1(D_S) \le \frac{6}{\pi^2 n} \sum_{M\in\{S_a,S_b\}^{2n}} \min_{\mathbf{x}\in \Delta_M} \frac{1}{x_1} \min_{\mathbf{x}\in \Delta_M} \frac{1}{x_2} \mathrm{Leb}(\Delta_M) \max_{\mathbf{x}\in \Delta_M} \log \| D_S^{(2n)}(\mathbf{x}) \|_\infty.
\end{equation}

As noted in \cite[Lemma~4.5]{HK:02}, the function $\mathbf{x} \mapsto \|D_S^{(2n)}(\mathbf{x})\|_\infty$ is convex on~$\Delta_M$, hence, the maximum $ \max_{\mathbf{x}\in \Delta_M} \log \| D_S^{(2n)}(\mathbf{x}) \|_\infty$ is attained in one of the corners of~$\Delta_M$. This makes \eqref{eq:tcd2} amenable for estimating $\lambda_1(D_S)$ with help of computer calculations. 
Indeed, taking $n=25$ in \eqref{eq:tcd2} we gain $\lambda_2(A_S) = \lambda_1(D_S) < -0.052435991$. We refer to the appendix for details on how we handle the numerical issues of this computer calculation.
\end{proof}

In view of Proposition~\ref{prop:PU} we can formulate a result that is true uniformly for \emph{all} $\mathbf{x}\in\Delta$.

\begin{proposition}\label{prop:selmerPU}
For the Selmer algorithm with $d=2$  there  exists $C>0$ such that for all  ${\bf x}$ and  all  $i,j\in\{1,\ldots,d\}$ we have $| p_{i,j}^{(n)} - q_i ^{(n)}  x_j |  \leq C$. Moreover, the inequality
\[
\Vert \wedge^2 A ^{(n)}(\mathbf{x}) \Vert \ll \Vert A ^{(n)}(\mathbf{x}) \Vert
\]
holds. Here the implied constant does not depend on $\mathbf{x}$ and $n\in\mathbb{N}$. 
\end{proposition}

\begin{proof}
In the proof of Theorem~\ref{thm:strongSelmer} we showed that $\Vert D^{(2n)}(\mathbf{x})\Vert \le 1$. By submultiplicativity this implies that $\Vert D^{(n)}(\mathbf{x})\Vert \ll 1$ holds uniformly for all $\mathbf{x}\in\Delta$ and all $n\in\mathbb{N}$. The result thus follows from Proposition~\ref{prop:PU}.
\end{proof}

Avila and Delecroix~\cite{AD15} proved that primitive Brun matrices for $d=2$ and primitive Arnoux--Rauzy matrices with $d \ge 2$ are Pisot, i.e., all eigenvalues except the Perron--Frobenius eigenvalue have absolute value less than~$1$.
We prove the analogous result for Selmer with $d=2$. 

\begin{theorem}
Let $d=2$ and $M \in \{S_a,S_b\}^n$ for some $n \ge 1$.
The following are equivalent.
\begin{enumerate}
\item
$M$ is a primitive matrix.
\item
$M$ is a Pisot matrix.
\item
$M^2 \not\in \{S_a S_b, S_b^2\}^n \cup \{S_b S_a, S_b^2\}^n$.
\end{enumerate}
\end{theorem}

\begin{proof}
Let first $M \in \{S_a,S_b\}^n$ be a primitive matrix. By taking a suitable power of $M$ if necessary, we may assume w.l.o.g.\ that $M$ is a positive matrix. Let $\bv\in\Lambda$ be the left eigenvalue of $M$ corresponding to the Perron--Frobenius eigenvalue and set $(v_1,v_2)=\kappa(\bv)$. Then from \eqref{eq:DnAn} (cf.~\cite[Section~3]{HK:02}) we easily derive that, up to a change of basis, $D^{(2n)}(v_1,v_2)$ is the restriction of~$M^2$ to~$\bv^\bot$. 
We gain from the proof of Theorem~\ref{thm:strongSelmer} that 
\begin{equation}\label{eq:Dbd}
\Vert D^{(2n)}(x_1,x_2)\Vert_\infty \le 1 \quad \hbox{ for each } (x_1,x_2) \in \Delta_{M^{2}}. 
\end{equation}
Since $\Vert\cdot\Vert_\infty$ is a consistent matrix norm this implies that each eigenvalue of $M$, except its Perron--Frobenius eigenvalue, has modulus less than or equal to $1$. 

Suppose that $M$ and, hence, $M^{2}=(m_{i,j})_{0\le i,j\le 2}$ has an eigenvalue of modulus $1$. Then, by the compatibility of the norm, we have $\Vert D^{(2n)}(v_1,v_2)\Vert_\infty = 1$. 
By \eqref{eq:Dn} we have
\begin{equation}\label{eq:D2d}
D^{(2n)}(x_1,x_2) = \begin{pmatrix}m_{1,1} - m_{1,0} x_1  & m_{1,2} - m_{1,0} x_2 \\  m_{2,1} - m_{2,0} x_1& m_{2,2} - m_{2,0} x_2\end{pmatrix} \quad \hbox{for each }(x_1,x_2)\in  \Delta_{M^{2}}.
\end{equation}
Since $M$ is positive, by the definition of $\Delta_M$ in \eqref{eq:DeltaM} the point $(v_1,v_2)$ is contained in the interior of $\Delta_{M^{2}}$; indeed, a positive matrix maps each (closed) positive cone into its interior. Let $U\subset\Delta_{M^{2}}$ be a neighborhood of $(v_1,v_2)$. Since $\Vert D^{(2n)}(v_1,v_2)\Vert_\infty = 1$, we see from \eqref{eq:D2d} and the definition of $\Vert \cdot \Vert_\infty$ (noting that the entries $m_{i,0}$ of $M^2$ are nonzero for $1\le i\le 2$) that there is $(x_1,x_2)\in U$ with $\Vert D^{(2n)}(x_1,x_2)\Vert_\infty > 1$, a contradiction to \eqref{eq:Dbd}. 
Thus, save for the Perron--Frobenius eigenvalue, each eigenvalue of $M$ has modulus less than $1$. Since $M$ is regular, this entails that the characteristic polynomial of $M$ is the minimal polynomial of a Pisot number, hence, $M$ is a Pisot matrix.

Conversely, it is well known that Pisot matrices are primitive (see e.g.\ \cite[Theorem~1.2.9]{Fog02}), i.e., we have (1) $\Leftrightarrow$ (2). 

If $M^2 \in \{S_a S_b, S_b^2\}^n$, then the first line of $M^{2k}$ equals $(1,0,0)$ for all $k \ge 1$, hence, $M$ is not primitive (and $1$ is an eigenvalue of~$M$).
Similarly, for each product of the matrices $S_b S_a$ and $S_b^2$, the second line equals $(0,1,0)$, hence $M$ is not primitive if $M^2 \in \{S_b S_a, S_b^2\}^n$.
Finally, when $M^2 \not\in \{S_a S_b, S_b^2\}^n \cup \{S_b S_a, S_b^2\}^n$, then $M^2$ contains a product of  the form $S_a S_b^{2k} S_a$ for some $k \ge 0$.
Note that the diagonals of $S_b^2$,  $(S_a S_b)^2$ and $(S_b S_a)^2$ are positive, hence multiplying a nonnegative matrix by one of these matrices does not decrease any of its elements.
Therefore, we find that $M^5$ contains a factor that is at least as large as $S_a^5$, $S_a^4 S_b$, $S_a^3 S_b S_a$ or $(S_a^2 S_b)^2$, which are all positive matrices. 
This shows that $M$ is primitive, thus (1) $\Leftrightarrow$ (3). 
\end{proof}

\subsection{Second Lyapunov exponent, $d = 3$}
In this  case the situation is more intricate than for $d=2$. 
Firstly, $S_a$ has now a pair of complex eigenvalues outside the unit circle, hence, $\Vert D_S^{(n)}(\mathbf{x})\Vert \le 1$ cannot hold for all $\mathbf{x} \in \Delta_{S_a} \cup \Delta_{S_b}$.    
Secondly, the conjectured value of $\lambda_2(A_S)$ is approximately $-0.02283$ (see Table~\ref{tab:0}) and therefore much closer to zero than in the case $d=2$. Nevertheless, we are able to establish the following convergence result.

\begin{theorem} \label{thm:strongSelmer3}
For $d=3$, the second Lyapunov exponent of the Selmer algorithm satisfies
\[
\lambda_2(A_S) < -0.000436459.
\]
In particular, for $d=3$ the Selmer algorithm is a.e.\ strongly convergent.
\end{theorem}

\begin{proof}
In the same way as in the proof of Theorem~\ref{thm:strongSelmer} we derive the estimate
\[
\lambda_2(A_S)=\lambda_1(D_S) \le \frac{1}{52} \sum_{M\in\{S_a,S_b\}^{52}} \mu_S(\Delta_M) \max_{\mathbf{x}\in \Delta_M} \log \|D_S^{(52)}(\mathbf{x})\|_\infty.
\]
However, since $\max_{\mathbf{x}\in \Delta_M} \log \|D_S^{(52)}(\mathbf{x})\|_\infty$ can be positive as well as negative we have to split this sum accordingly. In particular, we write
\[
\begin{split}
\lambda_1(D_S) \le \frac{1}{52} & \bigg( 
 \mu_S(\Delta_{S_b^{52}}) \max_{\mathbf{x}\in\Delta_{S_b^{52}}} \log \|D_S^{(52)}(\mathbf{x})\|_\infty \\
& + {\sum}^+ \mu_S(\Delta_M) \max_{\mathbf{x}\in \Delta_M} \log \|D_S^{(52)}(\mathbf{x})\|_\infty  +
{\sum}^- \mu_S(\Delta_M) \max_{\mathbf{x}\in \Delta_M} \log \|D_S^{(52)}(\mathbf{x})\|_\infty
\bigg).
\end{split}
\]
Here $\sum^+$ ranges over all $M\in\{S_a,S_b\}^{52}\setminus\{S_b^{52}\}$ satisfying $\max_{\mathbf{x}\in \Delta_M} \log \|D_S^{(52)}(\mathbf{x})\|_\infty \ge~0$ and $\sum^-$ ranges over all $M\in\{S_a,S_b\}^{52}\setminus\{S_b^{52}\}$ satisfying $\max_{\mathbf{x}\in \Delta_M} \log \|D_S^{(52)}(\mathbf{x})\|_\infty <~0$. The summand corresponding to $M=S_b^{52}$ has to be treated separately because the density of $\mu_S$ is not bounded in $\Delta_{S_b^{52}}$. We now use the estimates 
\[
c \min_{\mathbf{x}\in \Delta_M} \frac{1}{x_1} \min_{\mathbf{x}\in \Delta_M} \frac{1}{x_2} \mathrm{Leb}(\Delta_M)
\le
\mu_S(\Delta_M)
\le
c \max_{\mathbf{x}\in \Delta_M} \frac{1}{x_1} \max_{\mathbf{x}\in \Delta_M} \frac{1}{x_2} \mathrm{Leb}(\Delta_M),
\]
where, in view of \eqref{eq:mu}, we have
\[
c = \bigg(\int_{\Delta_{S_a} \cup \Delta_{S_b}}  \frac{\mathrm{d}x_1\mathrm{d}x_2\mathrm{d}x_3}{x_1x_2x_3} \bigg)^{-1} = \frac{8}{\zeta(3)},
\]
with $\zeta(s)$ being the Riemann zeta function. We therefore arrive at
\begin{equation}\label{eq:toCompute}
\begin{split}
\lambda_1(D_S) \le \frac{1}{52}\bigg( &
 \mu_S(\Delta_{S_b^{52}}) \max_{\mathbf{x}\in \Delta_{S_b^{52}}} \log \|D_S^{(52)}(\mathbf{x})\|_\infty 
 \\
 &+
\frac{8}{\zeta(3)}{\sum}^+ \max_{\mathbf{x}\in \Delta_M} \frac{1}{x_1} \max_{\mathbf{x}\in \Delta_M} \frac{1}{x_2} \mathrm{Leb}(\Delta_M) \max_{\mathbf{x}\in \Delta_M} \log \|D_S^{(52)}(\mathbf{x})\|_\infty  
\\
& +
\frac{8}{\zeta(3)}{\sum}^- \min_{\mathbf{x}\in \Delta_M} \frac{1}{x_1} \min_{\mathbf{x}\in \Delta_M} \frac{1}{x_2} \mathrm{Leb}(\Delta_M) \max_{\mathbf{x}\in \Delta_M} \log \|D_S^{(52)}(\mathbf{x})\|_\infty
\bigg).
\end{split}
\end{equation}
The right hand side of \eqref{eq:toCompute} can be bounded from above by $-0.000436459$ using extensive computer calculations. This yields the result. Details on the computer calculations are given in the appendix; we note already here that $-0.000436459$ is really an upper bound for $\lambda_2(A_S)$ because our programs are provided with an appropriate handling of the occurring floating point errors.  
\end{proof}

Note that $\Vert D_S^{(n)}(\mathbf{x})\Vert$ is not bounded by $1$. Also, there is no reason  for  \eqref{eq:Dnbound} and  a Paley--Ursell inequality to  hold, although the algorithm $A_S$ satisfies $\lambda_2 (A_S)<0$.

\subsection{Second Lyapunov exponent, $d\ge 4$}

Recall that, for arbitrary dimension $d$, the cocyle $D_S^{(n)}(\mathbf{x})$ is given by
\[
D_S^{(1)}(\mathbf{x}) = \begin{pmatrix}0&1&0&\cdots&0 \\ \vdots&\ddots&\ddots&\ddots&\vdots \\ 0&\cdots&0&1&0 \\ -x_1&-x_2&\cdots&-x_{d-1}&1-x_d \\ -x_1&-x_2&\cdots&-x_{d-1}&-x_d\end{pmatrix} \quad \mbox{if}\ \mathbf{x} \in \Delta_{S_a},
\]
and the last two lines are exchanged for $\mathbf{x} \in \Delta_{S_b}$.
(In dimension $d=2$, we have $D_S^{(1)}(\mathbf{x}) =\small \begin{pmatrix}-x_1&1-x_2\\-x_1&-x_2\end{pmatrix}$ if $\mathbf{x} \in \Delta_{S_a}$, $D_S^{(1)}(\mathbf{x}) =\small \begin{pmatrix}-x_1&-x_2\\-x_1&1-x_2\end{pmatrix}$ if $\mathbf{x} \in \Delta_{S_b}$.) 
Evaluating $\frac{1}{n} \log \|A_S^{(n)}(\mathbf{x})\|$ and $\frac{1}{n} \log \|D_S^{(n)}(\mathbf{x})\|$ for randomly chosen points $\mathbf{x}$ and $n=2^{30}$ gives the estimates listed in Table~\ref{tab:0} for $\lambda_1(A_S)$ and $\lambda_1(D_S) = \lambda_2(A_S)$ (without guaranteed accuracy; compare \cite{Lab15} for the value in the case $d=2$). 
See the end of the Appendix for details on the computation. 

\begin{table}[h]
\begin{tabular}{c|c|c}
$d$ & $\lambda_2(A_S)$ & $1-\frac{\lambda_2(A_S)}{\lambda_1(A_S)}$ \\ \hline 
$2$ & $-0.07072$ & $1.3871$\\
$3$ & $-0.02283$ & $1.1444$ \\
$4$ & $+0.00176$ & $0.9866$\\
$5$ & $+0.01594$ & $0.8577$
\end{tabular}
\medskip
\caption{Heuristically estimated values for the second Lyapunov exponent and the uniform approximation exponent of the Selmer Algorithm\label{tab:0}}
\end{table}

\subsection{Cassaigne algorithm}
In 2015, Cassaigne defined an (unordered) continued fraction algorithm that was first studied in~\cite{CLL:17,AL18} where it was shown to be conjugate to Selmer's algorithm. 
 The  motivation for defining this new algorithm came from word combinatorics. Define the two matrices
\[
C_a = \begin{pmatrix}1&0&0\\1&0&1\\0&1&0\end{pmatrix}, \qquad C_b = \begin{pmatrix}0&1&0\\1&0&1\\0&0&1\end{pmatrix},
\]
set $\Delta' = \{(x_0,x_1,x_2) \in \mathbb{R}_+^3:\, x_0+x_1+x_2=1\}$, and
\[
A_C:\, \Delta' \to GL(3,\mathbb{Z}), \quad \mathbf{x} \mapsto \begin{cases}C_a & \mbox{if}\ \mathbf{x} \in \Delta'_{C_a} = \{(x_0,x_1,x_2) \in \Delta':\, x_0>x_2\}, \\ C_b & \mbox{if}\ \mathbf{x} \in \Delta'_{C_b} = \{(x_0,x_1,x_2) \in \Delta':\, x_0<x_2\}.\end{cases}
\]
Then the Cassaigne map is
\[
T_C:\, \Delta' \to \Delta' \quad \mbox{defined by} \quad T_C(\mathbf{x})= \frac{ \mathbf{x}\, A_C(\mathbf{x})^{-1} } {\Vert\mathbf{x}\, A_C(\mathbf{x})^{-1}\Vert_1}.
\]
From \cite[Section~5]{CLL:17}, we know that the Cassaigne algorithm is conjugate to the semi-sorted Selmer algorithm (defined e.g.~in \cite[Section~4]{CLL:17}) on the absorbing set, which differs from the sorted version of the Selmer algorithm defined in Section~\ref{sec:defselmer} only by the order of the elements. Therefore, all these algorithms have the same Lyapunov spectrum. 

\section{Brun and modified Jacobi--Perron algorithms}\label{sec:Brun}
For the homogeneous version of the Brun algorithm~\cite{Brun19,Brun20,BRUN}, the second largest element of a vector is subtracted from the largest one and the resulting vector is ordered descendingly, i.e., for its projective version we have
\[
T_B:\, \Delta \to \Delta, \quad T_B(x_1,\ldots,x_d) = \kappa(\mathrm{ord}(1-x_1,x_1,x_2,\ldots,x_d))
\]
with $\kappa$ as in \eqref{eq:iotakappa}.
To get the associated matrix valued function $A_B$, we define 
\[
B_0 = \begin{pmatrix}1&0&\cdots&\cdots&0 \\ 1&1&\ddots&&\vdots \\ 0&0&\ddots&\ddots&\vdots \\ \vdots&\vdots&\ddots&1&0 \\ 0&0&\cdots&0&1\end{pmatrix}, \quad
B_k = \begin{pmatrix}\hspace{3em} 1&1&0&\!\cdots\!&\!\cdots\!&\!\cdots\!&\!\cdots\!&0 \\[-1ex] \hspace{3em} 0&0&1&\!\ddots\!&&&&\vdots \\[-2ex] k\!-\!1 \Bigg\{\ \vdots&0&\!\ddots\!&\!\ddots\!&\!\ddots\!&&&\vdots \\[-3ex] \hspace{3em} 0&\vdots&\!\ddots\!&\!\ddots\!&1&\!\ddots\!&&\vdots \\[-1ex] \hspace{3em} 1&\vdots&&\!\ddots\!&0&0&\!\ddots\!&\vdots \\[-1ex] \hspace{3em} 0&\vdots&&&\!\ddots\!&1&\!\ddots\!&0 \\[-2ex] d\!-\!k \Bigg\{\ \vdots&\vdots&&&&\!\ddots\!&\!\ddots\!&0 \\[-2ex] \hspace{3em} 0&0&\!\cdots\!&\!\cdots\! &\!\cdots\!&\!\cdots\!&0&1\end{pmatrix},\ 1\le k \le d.
\]
Setting $x_0=1$, $x_{d+1}=0$, and
\[
\Delta_{B_k} = \{(x_1,\ldots,x_d) \in \Delta:\, x_{k+1}<1-x_1<x_k\} \qquad(0\le k\le d)
\]
we have
\[
A_B(\mathbf{x}) = B_k \quad \mbox{if}\ \mathbf{x} \in \Delta_{B_k} \qquad(0\le k\le d).
\]
In view of \cite[Section~6]{Lagarias:93}, Brun's algorithm satisfies the assumptions of Proposition~\ref{prop:lag} (in particular, it satisfies the assumptions (H1) to (H5)). 
Evaluating $\frac{1}{n} \log \|A_B^{(n)}(\mathbf{x})\|$ and $\frac{1}{n} \log \|D_B^{(n)}(\mathbf{x})\|$ for randomly chosen points $\mathbf{x}$ and $n=2^{30}$ gives the estimates listed in Table~\ref{tab:1} for $\lambda_1(D_B) = \lambda_2(A_B)$ and for the uniform approximation exponent. 

\begin{table}[ht]
\begin{tabular}{c|c|c||c|c|c}
$d$ & $\lambda_2(A_B)$ & $1-\frac{\lambda_2(A_B)}{\lambda_1(A_B)}$ &
$d$ & $\lambda_2(A_B)$ & $1-\frac{\lambda_2(A_B)}{\lambda_1(A_B)}$ \\ \hline
$2$ & $-0.11216$ & $1.3683$ & $7$ & $-0.01210$ & $1.0493$ \\
$3$ & $-0.07189$ & $1.2203$ & $8$ & $-0.00647$ & $1.0283$ \\
$4$ & $-0.04651$ & $1.1504$ & $9$ & $-0.00218$ & $1.0102$ \\
$5$ & $-0.03051$ & $1.1065$ & $10$ & $+0.00115$ & $0.9943$ \\
$6$ & $-0.01974$ & $1.0746$ & $11$ & $+0.00381$ & $0.9799$ 
\end{tabular}
\medskip
\caption{Heuristically estimated values for the second Lyapunov exponent and the uniform approximation exponent of the Brun Algorithm\label{tab:1}}
\end{table}

The modified Jacobi--Perron algorithm (or $d$-dimensional Gauss algorithm), which goes back to Podsypanin~\cite{P:77}, is an accelerated version of the Brun algorithm, defined by the jump transformation $\mathbf{x} \mapsto T_B^n(\mathbf{x})$ with the minimal $n \ge 1$ such that $T_B^{n-1}(\mathbf{x}) \notin \Delta_{B_0}$; see \cite[Section~6.2]{Schweiger:00}. 
Its second Lyapunov exponent is thus negative if and only if $\lambda_2(A_B) < 0$.
In particular, the conjecture of \cite{Hardcastle:02} that the second Lyapunov exponent is negative for all $d \ge 2$ seems to be wrong in view of Table~\ref{tab:1}. 
We mention that for $d=2$ negativity of $\lambda_2(A_B)$ is proved in \cite{IKO,FIKO} by heavy use of computer calculation. Later, Meester~\cite{M:99} found a more elegant proof by deriving a Paley--Ursell type inequality for this setting and adapting Schweiger's argument from~\cite[Chapter 16]{Schweiger:00}.
Avila and Delecroix~\cite{AD15} gave a simple proof by showing that the $\infty$-norm of the restriction of $A_B^{(n)}(\mathbf{x})$ to $\iota(\mathbf{x})^\bot$ is bounded by~$1$; see Remark~\ref{r:AD}. Schratzberger~\cite{Schratzberger:01} gave  a  proof of the  strong convergence of  Brun algorithm in dimension $d=3$. Hardcastle~\cite{Hardcastle:02} even shows that  $\lambda_2(A_B) < 0$ holds for $d=3$. The dependence of the entropy of  the Brun algorithm with respect to the dimension is  studied in \cite{BLV:18}.

\section{Jacobi--Perron algorithm}\label{sec:JP}
We now consider the Jacobi--Perron algorithm; see \cite[Chapter~4 and 16]{Schweiger:00}, earlier references are \cite{Bernstein:71,Schweiger:73}. A~projective version of this algorithm is given by 
\[
T_J:\ [0,1]^d \to [0,1]^d, \quad 
(x_1,x_2,\ldots,x_d) \mapsto \Big(\frac{x_2}{x_1} - \Big\lfloor\frac{x_2}{x_1}\Big\rfloor, \ldots, \frac{x_d}{x_1} - \Big\lfloor\frac{x_d}{x_1}\Big\rfloor, \frac{1}{x_1} - \Big\lfloor\frac{1}{x_1}\Big\rfloor\Big).
\]
Its matrix version is therefore
\[
(x_0,x_1,\ldots,x_d) \mapsto \Big(x_1, x_2 - \Big\lfloor\frac{x_2}{x_1}\Big\rfloor x_1, \ldots, x_d - \Big\lfloor\frac{x_d}{x_1}\Big\rfloor x_1, x_0 - \Big\lfloor\frac{x_0}{x_1}\Big\rfloor  x_1\Big),
\]
and we have
\[
A_J(x_1,\ldots,x_d) = \begin{pmatrix}\lfloor\frac{1}{x_1}\rfloor & 1 & \lfloor\frac{x_2}{x_1}\rfloor & \cdots & \lfloor\frac{x_{d-1}}{x_d}\rfloor \\ 0&0&1&0&0 \\ \vdots&\vdots&\ddots&\ddots&0 \\ 0&\vdots&&\ddots&1 \\ 1&0&\cdots&\cdots&0\end{pmatrix}.
\]
This is a multiplicative algorithm in the sense that  divisions are performed  instead of subtractions, hence the coordinates are multiplied by arbitrarily large integers, and there are infinitely many different matrices~$A_J(\mathbf{x})$. 
It is proved in \cite[Section~5]{Lagarias:93} that the Jacobi--Perron algorithm satisfies the assumptions of Proposition~\ref{prop:lag} (in particular, it satisfies the assumptions (H1) to (H5)). The Jacobi--Perron algorithm is not ordered, thus it is defined in the whole unit cube.

It is known that the second Lyapunov exponent of the Jacobi--Perron algorithm is negative for $d=2$. A~proof of this fact, based on an old result by Paley and Ursell~\cite{PU:30}, is given in Schweiger~\cite[Chapter 16]{Schweiger:00}. 
Table~\ref{tab:JP} contains numerical estimates for the Lyapunov exponents of the Jacobi--Perron algorithm for low dimensions.
This table indicates that, like for the Brun algorithm, the second Lyapunov exponent of the Jacobi--Perron algorithm is negative for all $d \le 9$ and positive for all $d \ge 10$. This gives evidence that \cite[Conjecture~1.2]{Lagarias:93} does not hold.

\begin{table}[ht]
\begin{tabular}{c|c|c||c|c|c}
$d$ & $\lambda_2(A_J)$ & $1-\frac{\lambda_2(A_J)}{\lambda_1(A_J)}$ &
$d$ & $\lambda_2(A_J)$ & $1-\frac{\lambda_2(A_J)}{\lambda_1(A_J)}$ \\ \hline
$2$ & $-0.44841$ & $1.3735$ & $7$ & $-0.02819$ & $1.0243$ \\
$3$ & $-0.22788$ & $1.1922$ & $8$ & $-0.01470$ & $1.0127$ \\
$4$ & $-0.13062$ & $1.1114$ & $9$ & $-0.00505$ & $1.0044$ \\
$5$ & $-0.07880$ & $1.0676$ & $10$ & $+0.00217$ & $0.9981$\\
$6$ & $-0.04798$ & $1.0413$ & $11$ & $+0.00776$ & $0.9933$
\end{tabular}
\medskip
\caption{Heuristically estimated values for the second Lyapunov exponent and the uniform approximation exponent of the Jacobi--Perron Algorithm\label{tab:JP}}
\end{table}

\section{An intermediate algorithm between Arnoux--Rauzy and Brun}\label{sec:Algo}
From \cite{AD15}, we know that the second Lyapunov exponent of the Arnoux--Rauzy algorithm is negative for all $d \ge 2$, but this algorithm is only defined on a set of Lebesgue measure zero. 
We propose an algorithm that is in some sense between Arnoux--Rauzy and Brun: We subtract as many of the subsequent elements of a given vector from the first one (which is also the largest one) as possible.  (In the Arnoux--Rauzy algorithm, we always subtract all but the largest element from the largest one.) The matrix version of this algorithm is (with $x_{d+1}=x_0$)
\[
(x_0,x_1,\ldots,x_d) \mapsto \mathrm{ord}\Big( 
x_0 - \sum_{j=1}^k x_j,x_1,\ldots, x_d \Big) \quad \mbox{if}\ \sum_{j=1}^k x_j < x_0 < \sum_{j=1}^{k+1} x_j \quad (1\le k\le d).
\]
Denote  by $\Delta_{I_{k,\ell}}$, $1 \le k < d$, $k \le \ell \le d$, the set of $(x_1,\ldots,x_d) \in \Delta$ with $\sum_{j=1}^k x_j < 1 < \sum_{j=1}^{k+1} x_j$ and $x_\ell > 1 - \sum_{j=1}^k x_j > x_{\ell+1}$ (where $x_{d+1} = 0$), and denote by $\Delta_{I_{d,\ell}}$, $0 \le \ell \le d$, the set of $(x_1,\ldots,x_d) \in \Delta$ with $\sum_{j=1}^d x_j < 1$ and $x_\ell > 1 - \sum_{j=1}^d x_j > x_{\ell+1}$ (where $x_0 = 1$, $x_{d+1} = 0$).
Then we have
\[
A_I(\mathbf{x}) = I_{k,\ell} \quad \mbox{if}\ \mathbf{x} \in \Delta_{I_{k,\ell}},
\]
with 
\[
I_{k,\ell} = \begin{pmatrix}\hspace{3em} 1&1&0&\!\cdots\!&\!\cdots\!&\!\cdots\!&\!\cdots\!&\!\cdots\!&\!\cdots\!&0 \\[-2ex] \hspace{1.5em} k \Bigg\{\ \vdots&0&1&\!\ddots\!&&&&&&\vdots \\[-3ex] \hspace{3em} 1&0&\!\ddots\!&\!\ddots\!&\!\ddots\!&&&&&\vdots \\[-1ex] \hspace{3em} 0&\vdots&\!\ddots\!&\!\ddots\!&\!\ddots\!&\!\ddots\!&&&&\vdots \\[-2ex] \ell\!-\!k \Bigg\{\ \vdots&\vdots&&\!\ddots\!&\!\ddots\!&\!\ddots\!&\!\ddots\!&&&\vdots \\[-3ex] \hspace{3em} 0&\vdots&&&\!\ddots\!&\!\ddots\!&1&\!\ddots\!&&\vdots \\[-1ex] \hspace{3em} 1&\vdots&&&&\!\ddots\!&0&0&\!\ddots\!&\vdots \\[-1ex] \hspace{3em} 0&\vdots&&&&&\!\ddots\!&1&\!\ddots\!&0 \\[-2ex] d\!-\!\ell \Bigg\{\ \vdots&\vdots&&&&&&\!\ddots\!&\!\ddots\!&0 \\[-2ex] \hspace{3em} 0&0&\!\cdots\!&\!\cdots\!&\!\cdots\!&\!\cdots\!&\!\cdots\!&\!\cdots\!&0&1\end{pmatrix} \quad \mbox{if}\ 1\le k\le \ell \le d, 
\]
\[
I_{d,\ell} = \begin{pmatrix}\hspace{3em} 1&1&0&\!\cdots\!&\!\cdots\!&\!\cdots\!&0 \\[-2ex] \hspace{1.5em} \ell \Bigg\{\ \vdots&0&\!\ddots\!&\!\ddots\!&&&\vdots \\[-3ex] \hspace{3em} 1&0&\!\ddots\!&1&\!\ddots\!&&\vdots \\[-1ex] \hspace{3em} 1&\vdots&\!\ddots\!&0&0&\!\ddots\!&\vdots \\[-1ex] \hspace{3em} 1&\vdots&&\!\ddots\!&1&\!\ddots\!&0 \\[-2ex] d\!-\!\ell \Bigg\{\ \vdots&\vdots&&&\!\ddots\!&\!\ddots\!&0 \\[-2ex] \hspace{3em} 1&0&\!\cdots\!&\!\cdots\!&\!\cdots\!&0&1\end{pmatrix} \quad \mbox{if}\ 0\le \ell \le d.
\]
The Arnoux--Rauzy algorithm is the special case where $T^n\mathbf{x} \in \Delta_{I_{d,\ell}}$, $0 \le \ell \le d$, for all $n \ge 0$. 
It seems that the second Lyapunov exponent of our intermediate algorithm is negative for all $d \le 10$ and positive for all $d \ge 11$. 
The according heuristic estimates are listed in Table~\ref{tab:N}.

\begin{table}[ht]
\begin{tabular}{c|c|c||c|c|c}
$d$ & $\lambda_2(A_I)$ & $1-\frac{\lambda_2(A_I)}{\lambda_1(A_I)}$ &
$d$ & $\lambda_2(A_I)$ & $1-\frac{\lambda_2(A_I)}{\lambda_1(A_I)}$ \\ \hline
$2$ & $-0.13648$ & $1.3606$ & $7$ & $-0.02033$ & $1.0729$ \\
$3$ & $-0.10803$ & $1.2430$ & $8$ & $-0.01175$ & $1.0468$ \\
$4$ & $-0.07540$ & $1.1817$ & $9$ & $-0.00563$ & $1.0246$ \\
$5$ & $-0.05035$ & $1.1388$ & $10$ & $-0.00114$ & $1.0054$ \\
$6$ & $-0.03263$ & $1.1034$ & $11$ & $+0.00224$ & $0.9886$ 
\end{tabular}
\medskip
\caption{Heuristically estimated values for the second Lyapunov exponent and the uniform approximation exponent of the intermediate algorithm\label{tab:N}}
\end{table}

Using methods from Messaoudi, Nogueira, and Schweiger~\cite{MNS:09} as well as from Fougeron and Skripchenko \cite{FS:19} one can show that  the assumptions of Proposition~\ref{prop:lag} hold also for this algorithm. This will imply that negativity of the second Lyapunov exponent is a sufficient condition for strong convergence also for this algorithm. We will come back to this in a forthcoming paper.

\section{Garrity's triangle algorithm}\label{sec:Garrity}
A similar algorithm to the one in Section~\ref{sec:Algo} was proposed by Garrity~\cite{Garrity:01},  called the {\em triangle algorithm} (or {\em simplex algorithm for $d\ge 3$}), with the difference that the smallest coefficient is subtracted as many times as possible from the largest one when all other coefficients have already been subtracted. 
Similarly as  in the case of  Selmer's algorithm  (see \cite[Section 2]{Schweiger:04}), convergence properties are altered  by  taking  divisions  instead of  subtractions.
This will  be seen on the   second  Lyapunov exponent below.
Observe that this cannot be considered as  a real acceleration (as in the  regular   continued fraction case, or  as in  the Brun or in the Jacobi--Perron cases),
  since  taking divisions instead  of   subtractions
 yields a  completely  different  algorithm (similarly to the Selmer  case).

The matrix version of this algorithm is thus
\[
(x_0,x_1,\ldots,x_d) \mapsto \begin{cases}\mathrm{ord}\Big( 
x_0 - \sum_{j=1}^k x_j,x_1,\ldots, x_d \Big) & \\ & \hspace{-10em}\mbox{if}\ \sum_{j=1}^k x_j < x_0 < \sum_{j=1}^{k+1} x_j,\ 1\le k\le d-2, \\
\Big(x_1,\ldots, x_d,x_0 - \sum_{j=1}^{d-1} x_j - \ell x_d\Big) & \\ & \hspace{-10em}\mbox{if}\ \sum_{j=1}^{d-1} x_j + \ell x_d < x_0 < \sum_{j=1}^{d-1} (\ell+1) x_d,\ \ell \ge 0. 
\end{cases}
\]
We have 
\[
A_I(\mathbf{x}) = G_{k,\ell} \quad \mbox{if}\ \mathbf{x} \in \Delta_{G_{k,\ell}},
\]
with $G_{k,\ell} = I_{k,\ell}$ and $\Delta_{G_{k,\ell}} = \Delta_{I_{k,\ell}}$ for $1 \le k \le d-2$, $\ell \le k \le d$, 
\begin{gather*}
G_{d-1,\ell} = \begin{pmatrix}1&1&0&\cdots&0 \\ \vdots&0&\ddots&\ddots&\vdots \\ 1&\vdots&\ddots&1&0 \\ \ell&\vdots&&\ddots&1 \\ 1&0&\cdots&\cdots&0\end{pmatrix} \quad \mbox{for}\ \ell \ge 0, \\
\Delta_{G_{d-1,\ell}} = \Big\{(x_1,\ldots,x_d) \in \Delta:\, \sum_{j=1}^{d-1} x_j + \ell x_d < 1 < \sum_{j=1}^{d-1} (\ell+1) x_d\Big\}.
\end{gather*}
Here we have the curious situation that the second Lyapunov exponent seems to be negative if and only if $7 \le d \le 10$. 
The according heuristic estimates are listed in Table~\ref{tab:G}.

\begin{table}[ht]
\begin{tabular}{c|c|c||c|c|c}
$d$ & $\lambda_2(A_G)$ & $1-\frac{\lambda_2(A_G)}{\lambda_1(A_G)}$ &
$d$ & $\lambda_2(A_G)$ & $1-\frac{\lambda_2(A_G)}{\lambda_1(A_G)}$ \\ \hline
$2$ & $+0.34434$ & $0.6859$ & $7$ & $-0.00644$ & $1.0225$ \\
$3$ & $+0.37673$ & $0.5798$ & $8$ & $-0.00768$ & $1.0304$ \\
$4$ & $+0.25232$ & $0.6286$ & $9$ & $-0.00435$ & $1.0189$ \\
$5$ & $+0.10677$ & $0.7778$ & $10$ & $-0.00074$ & $1.0035$ \\
$6$ & $+0.01859$ & $0.9468$ & $11$ & $+0.00237$ & $0.9880$ 
\end{tabular}
\medskip
\caption{Heuristically estimated values for the second Lyapunov exponent and the uniform approximation exponent of Garrity's simplex algorithm\label{tab:G}}
\end{table}

Again using methods from \cite{MNS:09} and \cite{FS:19} one can show that  the assumptions of Proposition~\ref{prop:lag} hold also for this algorithm  in any dimension (although this is a bit more involved in this case because the algorithm is multiplicative).
The case $d=2$ is handled in \cite{FS:19}, and the general case will be addressed in a forthcoming paper. 

\section{Heuristical comparison between the algorithms} \label{sec:heuristic}

We conclude with a table that allows to compare the (heuristically estimated) uniform approximation exponents of the algorithms considered in this paper. In this table we also indicate Dirichlet's bound $ 1+1/d $.
\begin{table}[ht]
\begin{tabular}{c|c|c|c|c|c|l}
$d$ & Selmer & Brun & Jacobi--Perron & Intermediate & Garrity & $1+1/d$  \\ \hline
$2$ & $\mathbf{1.3871}$ & $1.3683$ & $1.3735$ & $1.3606$ & $0.6859$ & 1.5 \\
$3$ & $1.1444$ & $1.2203$ & $1.1922$ & $\mathbf{1.2430}$ & $0.5798$ &1.3333 \\
$4$ & $0.9866$ & $1.1504$ & $1.1114$ & $\mathbf{1.1817}$ & $0.6286$ & 1.25 \\
$5$ & $0.8577$ & $1.1065$ & $1.0676$ & $\mathbf{1.1388}$ & $0.7778$  & 1.2\\
$6$ & $0.7442$ & $1.0746$ & $1.0413$ & $\mathbf{1.1034}$ & $0.9468$  & 1.1667 \\
$7$ & $0.6437$ & $1.0493$ & $1.0243$ & $\mathbf{1.0729}$ & $1.0225$  & 1.1429\\
$8$ & $0.5561$ & $1.0283$ & $1.0127$ & $\mathbf{1.0468}$ & $1.0304$ & 1.125 \\
$9$ & $0.4810$ & $1.0102$ & $1.0044$ & $\mathbf{1.0246}$ & $1.0189$ & 1.1111 \\
$10$ & $0.4173$ & $0.9943$ & $0.9981$ & $\mathbf{1.0054}$ & $1.0035$ & 1.1\\
$11$ & $0.3636$ & $0.9799$ & $\mathbf{0.9933}$ & $0.9886$ & $0.9880$ & 1.0909
\end{tabular}
\medskip
\caption{Synopsis of the uniform approximation exponents $1 -\frac{\lambda_2(A)}{\lambda_1(A)}$}
\end{table}

\section*{Appendix: Comments on the floating point calculations}

In this appendix we discuss the computational issues of the calculations leading to the estimate of the Lyapunov exponent $\lambda_2(A_S)$ for the Selmer algorithm in Theorem~\ref{thm:strongSelmer} ($d=2$) and Theorem~\ref{thm:strongSelmer3} ($d=3$). As these calculations are extensive we had to execute them using a GPU. All calculations were performed on an \emph{Apple MacBook Pro 2019} with an \emph{Intel Iris Plus Graphics 655 1536 MB} card using the \emph{XCode} environment. The language we used is \emph{Objective C}, where the code executed on the GPU is implemented in Apple's \emph{Metal} language.  

We start with Selmer's algorithm for $d=3$; the easier case $d=2$ will be treated after that.
In order to estimate the second Lyapunov exponent $\lambda_2(A_S)$ of the Selmer algorithm, $d=3$, we use inequality \eqref{eq:toCompute}. Since $\bx \mapsto \log\Vert D_S^{52}(\bx)\Vert_\infty$ is convex (cf.~\cite[Lemma~4.5]{HK:02}), for each $M \in \{S_a,S_b\}^{52}$ it is sufficient to compare the values at the vertices of $\Delta_M$ to compute the maximum over $\Delta_M$ in \eqref{eq:toCompute}.

We first deal with the case $M \in \{S_a,S_b\}^{52} \setminus \{S_b^{52}\}$, i.e., with the sums $\sum^+$ and $\sum^-$ of \eqref{eq:toCompute}.
For each of the summands it is possible to calculate the rational numbers $\Vert D_S^{52}(\bx) \Vert_\infty$ for all $\bx$ being a vertex of $\Delta_M$ with $M \in \{S_a,S_b\}^{52}\setminus \{S_b^{52}\}$ by using integer arithmetics and treating the denominator and the numerator separately. Also $\mathrm{Leb}(\Delta_M)$, $\max_{\mathbf{x}\in \Delta_M} \frac{1}{x_1} \max_{\mathbf{x}\in \Delta_M} \frac{1}{x_2}$, and $\min_{\mathbf{x}\in \Delta_M} \frac{1}{x_1} \min_{\mathbf{x}\in \Delta_M} \frac{1}{x_2} \mathrm{Leb}(\Delta_M)$ can be calculated using integer arithmetics. Thus these calculations are exact. 

When taking the logarithm and multiplying it by $\max_{\mathbf{x}\in \Delta_M} \frac{1}{x_1} \max_{\mathbf{x}\in \Delta_M} \frac{1}{x_2} \mathrm{Leb}(\Delta_M)$ and  $\min_{\mathbf{x}\in \Delta_M} \frac{1}{x_1} \min_{\mathbf{x}\in \Delta_M} \frac{1}{x_2} \mathrm{Leb}(\Delta_M)$, respectively, we are forced to switch to floating point arithmetics. The software we use, namely \emph{Metal} and \emph{Objective C}, complies with the \emph{IEEE 754} standard for floating point arithmetics.\footnote{In \emph{Metal}, {\tt float} is the most precise data type for floating point calculations. For the part of the code written in \emph{Objective C} we use the data type {\tt long double} to gain higher precision.}
For relevant facts on floating point arithmetics and details on this IEEE standard, we refer e.g.\  to~\cite{Goldberg:91}; the language specification of \emph{Metal} is laid out in~\cite{Metal:17}.

Using floating point arithmetics entails rounding errors. Because we want an exact upper bound in the estimate of $\lambda_2(A_S)$ provided in Theorem~\ref{thm:strongSelmer3} we need to make sure that the error we produce by using floating point arithmetics yields a result which is not smaller than the exact result would be. To guarantee this, after each floating point operation we use the function\footnote{See for instance {\tt https://en.cppreference.com/w/c/numeric/math/nextafter} for a documentation of this function as well as its sibling {\tt long double nextafterl( long double x, long double y )}.} 

\smallskip
\noindent
\texttt{float nextafterf( float x, float y )}.
\smallskip

\noindent Setting \texttt{y=INFINITY} and \texttt{y=-INFINITY} this function returns the smallest floating point number which is greater than $\tt x$ and the largest floating point number which is smaller than $\tt x$, respectively.  Using this function makes the estimate for $\lambda_2(A_S)$ in Theorem~\ref{thm:strongSelmer3} exact (at the price that the modulus of the upper bound we gain is about 0.5\% to 1\% larger than it would be without applying this function). Our calculations yield 
\begin{equation}\label{eq:compA}
\begin{split}
&{\sum}^+ \max_{\mathbf{x}\in \Delta_M} \frac{1}{x_1} \max_{\mathbf{x}\in \Delta_M} \frac{1}{x_2} \mathrm{Leb}(\Delta_M) \max_{\mathbf{x}\in \Delta_M} \log \|D_S^{(52)}(\mathbf{x})\|_\infty  
\\
 + &
{\sum}^- \min_{\mathbf{x}\in \Delta_M} \frac{1}{x_1} \min_{\mathbf{x}\in \Delta_M} \frac{1}{x_2} \mathrm{Leb}(\Delta_M) \max_{\mathbf{x}\in \Delta_M} \log \|D_S^{(52)}(\mathbf{x})\|_\infty
\le -0.004845689,
\end{split}
\end{equation}
where $\sum^+$ and $\sum^-$ are defined as in \eqref{eq:toCompute}.

The summand in \eqref{eq:toCompute} corresponding to $M=S_b^{52}$ has to be treated separately as follows. First  note that  the estimate
\[
\mu_S(\Delta_{S_b^{52}})=\frac{8}{\zeta(3)} \int_{\Delta_{S_b^{52}}} \frac{\mathrm{d}x_1\mathrm{d}x_2\mathrm{d}x_3}{x_1x_2x_3} \le 0.004776713
\]
follows if one evaluates the integral using polylogarithms (which we did with the help of \emph{Mathematica}). Since it is easy to see that $\max_{\mathbf{x}\in \Delta_{S_b^{52}}} \log \|D_S^{(52)}(\mathbf{x})\|_\infty=2$ we gain
\begin{equation}\label{eq:compB}
 \mu_S(\Delta_{S_b^{52}}) \max_{\mathbf{x}\in \Delta_{S_b^{52}}} \log \|D_S^{(52)}(\mathbf{x})\|_\infty \le 0.009553426.
\end{equation}

Inserting \eqref{eq:compA} and \eqref{eq:compB} in \eqref{eq:toCompute} we end up with
\[
\lambda_2(A_S) \le -0.000436459,
\]
which is the upper bound for $\lambda_2(A_S)$ stated in Theorem~\ref{thm:strongSelmer3}.

To treat the case $d=2$ our starting point is \eqref{eq:tcd2} with $n=25$. Since we always have $\max_{\mathbf{x}\in \Delta_M} \log \|D_S^{(2n)}(\mathbf{x})\|_\infty \le 0$ the whole sum in \eqref{eq:tcd2} is of the type $\sum^-$ and, hence, also the contribution of $ \Delta_{S_b^{2n}}$ does not need to be treated separately. By the same strategy as the one outlined for $d=3$ we gain the estimate
\[
\sum_{M\in\{S_a,S_b\}^{2n}} \min_{\mathbf{x}\in \Delta_M} \frac{1}{x_1} \min_{\mathbf{x}\in \Delta_M} \frac{1}{x_2} \mathrm{Leb}(\Delta_M) \max_{\mathbf{x}\in \Delta_M} \log \| D_S^{(2n)}(\mathbf{x}) \|_\infty \le -2.06343104875.
\]
Inserting this in \eqref{eq:tcd2} yields the estimate stated in Theorem~\ref{thm:strongSelmer}.

\medskip
For the other dimensions and algorithms, we do not calculate upper bounds for the second Lyapunov exponent. 
Instead, we get heuristics for $\lambda_2(A)$ by calculating $D^{(n)}(\mathbf{x})$ for $n=2^{30}$ and ten randomly chosen points $\mathbf{x} \in \Delta$, using a \emph{C}~program with double precision floating point arithmetic.
In order for the matrices not to become too small or too large, we renormalize after each $k=2^{10}$ steps by dividing by the top left coefficient of the matrix. 
This means that we calculate iteratively 
\[
\frac{1}{D_{1,1}^{(jk+k)}(\mathbf{x})} D^{(jk+k)}(\mathbf{x}) = \frac{D_{1,1}^{(jk)}(\mathbf{x})}{D_{1,1}^{(jk+k)}(\mathbf{x})} D^{(k)}(T^{jk} \mathbf{x}) \, \frac{1}{D_{1,1}^{(jk)}(\mathbf{x})} D^{(jk)}(\mathbf{x}),
\]
for $0 \le j < n/k$, where $D_{1,1}^{(\ell)}(\mathbf{x})$ denotes the top left coefficient of the matrix $D^{(\ell)}(\mathbf{x})$.
Keeping track of the normalisation factors, we have 
\[
\log \big\|D^{(n)}(\mathbf{x})\big\| = \log \bigg\|\frac{D^{(n)}(\mathbf{x})}{D_{1,1}^{(n)}(\mathbf{x})}\bigg\| + \sum_{j=0}^{n/k-1} \log \bigg|\frac{D_{1,1}^{(jk+k)}(\mathbf{x})}{D_{1,1}^{(jk)}(\mathbf{x})}\bigg|.
\]

\bibliographystyle{amsalpha}
\bibliography{sadic5}
\end{document}